\pdfoutput=1
\documentclass[a4paper,11pt]{amsart}

\usepackage[utf8]{inputenc} 

\usepackage[T1]{fontenc}

\usepackage[english]{babel}

\usepackage{amsmath}
\usepackage{amsfonts}
\usepackage{amssymb}
\usepackage{amsthm} %
\usepackage{mathrsfs} %

\usepackage{subfig} %

\usepackage{tabularx} %
\newcolumntype{C}{>{\centering}X} %
\usepackage{calc} %
\usepackage[pdftex]{graphicx,color} %
\usepackage{hyperref}

\usepackage{multirow} %

\usepackage{xy} %
\usepackage{bm}
\xyoption{all}
\xyoption{poly}
\xyoption{arc}

\usepackage{a4wide} %

\usepackage{url} %
\usepackage{enumerate}

\usepackage{graphics}%

\usepackage{array} %

\theoremstyle{plain} %
\newtheorem{theo}{Theorem}[section]

\newtheorem{propo}[theo]{Proposition}
\newtheorem{lemma}[theo]{Lemma}

\theoremstyle{definition}
\newtheorem{defi}[theo]{Definition}

\newtheorem{remark}[theo]{Remark}

\newcommand{\CO}{{\mathcal{O}}} 
 
 \newcommand{\CH}{{\mathcal{H}}}

 \newcommand{\CA}{{\mathcal{A}}}

\newcommand{\BN}{{\mathbb{N}}}
\newcommand{\BC}{{\mathbb{C}}}

\newcommand{\fp}{{\mathfrak{p}}}
\newcommand{\fq}{{\mathfrak{q}}}
\newcommand{\FD}{{\mathfrak{D}}}

\makeatletter
\def\hlinewd#1{%
\noalign{\ifnum0=`}\fi\hrule \@height #1 %
\futurelet\reserved@a\@xhline}
\makeatother
  {\everymath{\displaystyle\everymath{}}\array}%
  {\endarray}

\newcommand{\qg}{{\backslash}} %

\newcommand{\Ubr}{{U_{\mathrm{branch}}}}
\newcommand{\Vram}{{V_{\mathrm{ram}}}}

\newcommand{\tq}{{~|~}}
\newcommand{\ie}{{\emph{i.e.}~}}
\newcommand{\surj}{{\twoheadrightarrow}}

\newcommand{\Tr}{\operatorname{Tr}\nolimits}
\newcommand{\GL}{\operatorname{GL}\nolimits}

\newcommand{\Spec}{\operatorname{Spec}\nolimits}

\renewcommand{\det}{\operatorname{det}\nolimits}

\newcommand{\Jac}{\operatorname{Jac}\nolimits}

\newcommand{\LL}{\operatorname{LL}\nolimits}

\newcommand{\Spram}{\operatorname{Spec_1^{\mathrm{ram}}}\nolimits}
\newcommand{\lcm}{\operatorname{lcm}\nolimits}

\newcommand{\grdim}{\operatorname{grdim}\nolimits}

\begin{document}

\title[Discriminants and Jacobians of virtual reflection
groups]{Discriminants and Jacobians \\of virtual reflection groups}

\author{Vivien Ripoll}
 \address{LaCIM, UQÀM, CP 8888, Succ. Centre-ville Montréal, QC, H3C
  3P8, Canada}
\email{vivien.ripoll@lacim.ca}

\begin{abstract}
  Let $A$ be a polynomial algebra with complex coefficients. Let $B$
  be a finite extension ring of $A$ which is also a polynomial
  algebra. We describe the factorisation of the Jacobian $J$ of the
  extension into irreducibles. We also introduce the notion of a
  well-ramified extension and define its discriminant polynomial
  $D$. In the particular case where $A$ is the ring of invariants of
  $B$ under the action of a group (\emph{i.e.}, a Galois extension), this
  framework corresponds to the classical invariant theory of complex
  reflection groups. In the more general case of a well-ramified
  extension, we explain how the pair $(D,J)$ behaves similarly to a
  Galois extension. This work can be viewed as the first step towards a
  possible invariant theory of ``virtual reflection groups''.
\end{abstract}

\maketitle

\section*{Introduction}

This note describes a non-Galois version of the first few steps of the
classical invariant theory of reflection groups. We will deal with
quite basic questions of commutative algebra, that were at first
motivated by empirical observations on the extensions defined by
Lyashko-Looijenga morphisms.

We consider a finite ring extension $A\subseteq B$, where $A$ and $B$
are both polynomial algebras in~$n$ variables. In the case when $A$ is
the ring of invariants of $B$ under the action of a group~$G$, this
implies that $G$ is generated by reflections
(Chevalley-Shephard-Todd's theorem) and many properties are known in
this setting. Here we do not suppose that $A$ has the form
$B^G$. Thus, we cannot simply imitate the classical proofs of
invariant theory, as they really make use of the group
action. However, in our setting, many properties seem to work the same
way as for Galois extensions, particularly for Jacobian and
``discriminant'' of the extension.

\bigskip

Note that we use only elementary commutative algebra, and that the
properties derived here are presumably folklore. The situation that we
describe is in fact surprisingly basic and universal, yet apparently
written nowhere from this perspective. The extensions usually studied in the
litterature are either much more general, or of the form $A=B^G \subseteq
B$, where~$B$ is a polynomial algebra; here we are rather interested in
extensions $A\subseteq B$ where~$B$ \emph{and}~$A$ are polynomial algebras,
but where we \emph{do not require} $A$ to be the ring of invariants of~$B$
under a group action.

\bigskip

The key ingredients to describe and understand the situation are:
\begin{itemize}
\item a notion of \emph{``well-ramified''} polynomial extensions;
\item properties of the \emph{different} of an extension, that enable to
  apprehend the Jacobian of the extension.
\end{itemize}

\subsection*{Motivations} 
Let $V$ be an $n$-dimensional complex vector space, and $W\subseteq
\GL(V)$ a finite complex reflection group, with fundamental system of
invariants $f_1, \dots, f_n$ of degrees $d_1 \leq \dots \leq
d_n$. From Chevalley-Shephard-Todd's theorem, we have the equality
$\BC[V]^W=\BC[f_1,\dots, f_n]$, and the isomorphism
\[ \begin{array}{lcl}
  W \qg V & \to & \BC^n \\
\bar{v} &\mapsto& (f_1(v),\dots, f_n(v)).
\end{array}\]

Let us denote by $\CA$ the set of all reflecting hyperplanes, and consider the
discriminant of $W$ defined by
\[ \Delta_W := \prod_{H\in \CA} \alpha_H^{e_H} \ ,\] where $\alpha_H$ is an
equation of $H$ and $e_H$ is the order of the parabolic subgroup $W_H$. The
discriminant is the equation of the hypersurface $\displaystyle{\CH:=W \qg \bigcup_{H\in
  \CA}H}$ in $\BC^n=\Spec \BC[f_1,\dots, f_n]$.

Let us also consider the Jacobian $J_W$ of the morphism $(v_1,\dots, v_n)
\mapsto (f_1(v),\dots, f_n(v))$:
\[ J_W := \det \left( \frac{\partial f_i}{\partial v_j}
\right)_{\substack{1 \leq i\leq n\\1\leq j \leq n}} .\] It is well known
(see for example \cite[Sect.\ 21]{Kane}) that the Jacobian satisfies the
following factorisation:
\[ J_W \doteq \prod_{H\in \CA} \alpha_H^{e_H -1}, \] where $\doteq$
denotes equality up to a nonzero scalar. Thus, we get \[\Delta_W / J_W
= \prod_{H\in \CA} \alpha_H \ ,\] \ie this quotient is the product of
the ramified polynomial of the extension $\BC[f_1,\dots, f_n]
\subseteq \BC[V]$.

\bigskip

A stunningly similar situation has came up recently, in the geometric
study of the Lyashko-Looijenga morphisms ($\LL$) associated to complex
reflection groups. These morphisms have been introduced by Bessis in
his work about the $K(\pi,1)$ property for finite complex reflection
arrangements \cite{BessisKPi1}; see also \cite{Ripollfacto} about the
relation with the combinatorics of factorisations of Coxeter
elements. The Lyashko-Looijenga morphism associated to a rank $n$
reflection group has the following form:
\[ \begin{array}{rccc}
\LL & \BC^{n-1} & \to & \BC^{n-1}\\
& X=(x_1,\dots, x_n) & \mapsto & (a_2(X),\dots, a_n(X))\ ,
\end{array}
\]
where $a_2,\dots, a_n$ are polynomials in $x_1,\dots, x_{n-1}$,
constructed from the geometry of the reflection group. We refer to
\cite[Sect.\ 5]{BessisKPi1} or \cite[Sect.\ 3]{Ripollfacto} for the
precise definition (which is not necessary in this note).

Associated to the morphism $\LL$, we can naturally define the
$\LL$-discriminant $D_{\LL}$ (it is the equation of the bifurcation
locus $\LL$), and the $\LL$-Jacobian $J_{\LL}$ (the Jacobian
determinant of the morphism). It turns out that the pair of
polynomials $(J_{\LL}, D_{\LL})$ behaves similarly to the pair~$(J_W,
\Delta_W)$ defined earlier: the quotient $D_{\LL}/J_{\LL}$ is the
product of the ramified polynomials of the extension associated to
$\LL$, and moreover their valuations in $D_{\LL}$ correspond to their
ramification indices.

\subsection*{General setting and main result}

Although we first had in mind applications to the Lyashko-Looijenga
context, it seems that a more general setting has its interest
itself. Thus, this note is devoted to the following general setup. Let
us consider a \emph{finite graded polynomial extension}~${A\subseteq
  B}$ (see Def.\ \ref{defext}): we have a graded polynomial
  algebra $B$ in $n$ indeterminates over $\BC$, and a polynomial
  subalgebra $A$ generated by $n$ weighted homogeneous elements of
  $B$, such that the extension is finite. The two key examples are the
  ones already mentioned:
\begin{itemize}
\item the Galois extensions $\BC[f_1,\dots, f_n]\subseteq
  \BC[v_1,\dots, v_n]$, defined by a morphism ${V \to W \qg V}$, where
  $w$ is a reflection group and $\BC[f_1,\dots, f_n]=\BC[V]^W$;
\item the Lyashko-Looijenga extensions $\BC[a_2,\dots, a_n]\subseteq
  \BC[x_1,\dots, x_{n-1}]$, given by a morphism $\LL$; these
  extensions are indeed finite according to \cite[Thm.\ 5.3]{BessisKPi1}.
\end{itemize}

In the first section we give the precise definitions, and use the notion of
\emph{different ideal} of an extension to describe a factorisation of the
Jacobian. In section \ref{partgeom}, about the geometry of such extensions,
we recall the relations between the ramification locus and the branch locus
of a branched covering. In section \ref{partwellram} we define the
\emph{well-ramified property} for a finite graded polynomial extension
(Def.\ \ref{defwellram}), and we give several characterisations of this
property (Prop.\ \ref{propwell}): this is a slightly weaker property than
the normality of the extension, and is also equivalent to the equality
between the preimage of the branch locus and the ramification locus.

\medskip

The main result of this chapter is:

\begin{theo}
  \label{thmintrojac}
  Let $\underline{W}=(A \subseteq B)$ be a finite graded polynomial
  extension. Then the Jacobian $J$ of the extension verifies:
  \[ J \doteq \prod_{Q \in \Spram(B)} Q^{e_Q-1}\] where $\Spram(B)$ is
  the set of ramified polynomials in $B$ (up to association), and the
  $e_Q$ are the associated ramification indices.

  Moreover, if the extension $\underline{W}$ is \emph{well-ramified}
  (according to Def.\ \ref{defwellram}), then:
  \[ (J)\cap A = \left( \prod_{Q \in \Spram(B)} Q^{e_Q} \right) \qquad
  \text{(as an ideal of } A\text{).}\]
\end{theo}

\begin{remark}
  The first part of Thm.\ \ref{thmintrojac} is a rather easy consequence of
  commutative algebra properties. It is probably folklore, but I could not
  find the formula stated anywhere.
\end{remark}

\begin{remark}
  The notation $\underline{W}$ for the extension is intentionally
  chosen to emphasize the analogy with the case when the extension is
  Galois, \ie when it is given by the action of a reflection group $W$
  on the polynomial algebra $B$. Here there is not necessarily a group
  acting, but some features of the Galois case remain. That is why
  David Bessis proposed to call the extension $\underline{W}$ a
  \emph{virtual reflection group}\footnote{David Bessis, personal
    communication.}. The polynomial $\prod Q^{e_Q}$ in the theorem
  above could then be called the \emph{discriminant}%
  of the virtual reflection group $W$. One can wonder whether the
  analogies can go further, and to what extent it is possible to
  construct an ``invariant theory'' for virtual reflection groups.
 \end{remark}

\section{Jacobian and different of a finite graded polynomial
  extension}

\subsection{General setting and notations}
~\\
\label{subpartsetting}
Let $n$ be a positive integer, and denote by $B$ the graded polynomial
algebra $\BC[X_1,\dots, X_n]$, where $X_1,\dots, X_n$ are indeterminates of
respective weights $b_1,\dots, b_n$.

Let us consider $n$ weighted homogeneous polynomials $f_1,\dots, f_n$ in $B$ (of
respective weights $a_1,\dots,a_n$), and the resulting
(quasi-homogenenous) mapping  
\[\begin{array}{lccc}
    f : & \BC^{n} & \to & \BC^{n}\\
& (x_1,\dots, x_n) & \mapsto & (f_1(x_1,\dots, x_n),\dots, f_n(x_1,\dots,
x_n)).
  \end{array} \]

We denote by $A$ the
algebra $\BC[f_1,\dots, f_n]$, so that we have a ring extension $A\subseteq B$.

\begin{defi}
\label{defext}
In the above situation, if $B$ is an $A$-module of finite type, we will
call~${A\subseteq B}$ a \emph{finite graded polynomial extension}.
\end{defi}

\begin{remark}
  In this setting, the extension is finite if and only if
  $f^{-1}(\{0\})=\{0\}$, because~$f$ is a quasi-homogeneous morphism
  (see for example \cite[Thm.5.1.5]{LZgraphs}). Moreover, the rank of
  $B$ over~$A$ (or the \emph{degree} of $f$) is then equal to
  $r:=\prod a_i / \prod b_i$.
\end{remark}

\bigskip

The algebra $B$ is Cohen-Macaulay, and is finite as an $A$-module, so $B$
is also a free~$A$-module of finite type.  Thus $A\subseteq B$ is a finite
free extension of UFDs. We denote by~$K$ and~$L$ the fields of fractions of
$A$ and $B$. Let us recall some notations and properties about the
ramification in this context (for example see \cite[Chap.\ 3]{Benson}).

\medskip

If $\fq$ is a prime ideal of $B$, then $\fp=\fq \cap A$ is a prime ideal of
$A$. In this situation we say that $\fq$ lies \emph{over} $\fp$. By the
Cohen-Macaulay theorem \cite[Thm.1.4.4]{Benson}, $\fq$ has height one if
and only if~$\fp$ has height one. In this case, we can write $\fq=(Q)$ and
$\fp=(P)$, where $P\in A$ and $Q \in B$ are irreducible, and $(Q)\cap
A=(P)$.

The ramification index of $\fq$ over $\fp$ is:
\[ e(\fq,\fp):= v_Q (P)\]
(we will rather write simply $e_\fq$ or $e_Q$).

\medskip

We denote by $\Spec_1(B)$ (resp. $\Spec_1(A)$) the set of prime ideals of
$B$ (resp. $A$) of height one. It is also the set of irreducible
polynomials in $B$, up to association. By abusing the notation, in indices
of products, we will write ``$Q \in \Spec_1(B)$'' instead of ``$\fq\in
\Spec_1(B)$ and~$Q$ is one polynomial representing $\fq$'', and ``$Q$ over
$P$'' instead of ``$(Q)$ over $(P)$''.

We recall the following elementary property:

\begin{lemma}
  Let $(P) \in \Spec_1(A)$ and $(Q)\in\Spec_1(B)$. Then, $(Q)$ lies
  over $(P)$ if and only if $Q$ divides $P$ in
  $B$. Thus, for $(P)\in\Spec_1(A)$, we have:
  \[ P \doteq \prod_{Q \text{\ over\ } P} Q^{e_Q} \ .\]
\end{lemma}

\subsection{Different ideal and Jacobian}

As our setup is compatible with that of \cite[Part.\ 3.10]{Benson}, we can
construct the \emph{different $\FD_{B/A}$} of the extension. Let us recall
that it is defined from the inverse different :
\[ \FD_{B/A}^{-1} := \{x \in L \tq \forall y \in B, \Tr_{L/K}(xy) \in A
\}\ ,\]
where $K$ and $L$ are the fields of fractions of $A$ and $B$, $L$ is
regarded as a finite vector space over~$K$, and $\Tr_{L/K}(u)$
denotes the trace of the endomorphism $(x\mapsto ux)$.

The different $\FD_{B/A}$ is then by definition the inverse fractional
ideal to $\FD_{B/A}^{-1}$. It is a (homogeneous) divisorial ideal; in our
setting, as $B$ is a UFD, $\FD_{B/A}$ is thus a principal ideal. We will
see below that the different is simply generated by the Jacobian $J_{B/A}$
of the extension. For now, let us denote by $\theta_{B/A}$ a homogeneous
generator of $\FD_{B/A}$.

\medskip

The different satisfies the following:

\begin{propo}[{\cite[Thm.3.10.2]{Benson}}]
\label{propcaracdiff}
If $\fq$ and $\fp=A\cap \fq$ are prime ideals of height one in~$B$ and $A$,
then $e(\fq,\fp)>1$ if and only if $\FD_{B/A} \subseteq \fq$.

In other words: if $Q$ is an irreducible polynomial in $B$, then $e_Q>1$ if
and only if $Q$ divides~$\theta_{B/A}$.
\end{propo}

We define the set of ramified ideals:
 \[ \Spram(B):=\{\fq \in \Spec_1(B) \tq e_\fq >1 \}, \]
 which can also be seen as a
system of representatives of the irreducible polynomials $Q$ in~$B$ which are
ramified over $A$. By the above theorem, we have:
\[ \theta_{B/A} \doteq \prod_{Q \in \Spram(B)} Q ^{v_Q(\theta_{B/A})} \ . \]

This can be refined as:

\begin{propo}
\label{propdiff}
For all irreducible $Q$ in $B$, we have: $v_Q(\theta_{B/A})=e_Q -1$. Thus:
\[ \theta_{B/A} \doteq \prod_{Q \in \Spram(B)} Q ^{e_Q -1}. \]
\end{propo}

\begin{proof}
  We localize at $(Q)$ in order to obtain local Dedekind domains. Then we
  can use directly Prop.\ 13 in \cite[Ch. III]{serre}. 
\end{proof}

Let us show that the different is actually generated by the Jacobian
determinant. For this, we need to introduce the \emph{Kähler different} of
the extension $A \subseteq B$. According to Broer \cite{diff} (end of first
section), when $B$ is a polynomial algebra, the Kähler different can be
defined as the ideal generated by the Jacobians of all $n$-tuples of
elements of $A$, with respect to $X_1,\dots,X_n$. But here we are in an
even more specific situation, where $A$ is also a polynomial ring. Thus,
whenever we take $g_1,\dots, g_n \in A=\BC[f_1,\dots,f_n]$, we have
\[ \det \left( \frac{\partial g_i}{\partial X_k} \right)_{\substack{1 \leq
    i\leq n\\1\leq k \leq n}} = \det \left( \frac{\partial g_i}{\partial f_j}
  \right)_{\substack{1 \leq i\leq n\\1\leq j \leq n}} \det \left( \frac{\partial
        f_j}{\partial X_k} \right)_{\substack{1 \leq j\leq n\\1\leq k \leq n}} \ ,\] so
      the Kähler different is simply the principal ideal of $B$ generated
      by the polynomial
\[ J_{B/A}:=\Jac ((f_1,\dots, f_n) / (X_1,\dots, X_n))= \det \left(
  \frac{\partial f_i}{\partial X_j} \right)_{\substack{1 \leq i\leq
    n\\1\leq j \leq n}} .\]

\begin{propo}
\label{propdiffjac}
With the hypothesis above, we have:
\[ \theta_{B/A} \doteq J_{B/A} .\]
\end{propo}

\begin{proof}
  In \cite{diff}, Broer studies several notions of different ideals, and
  proves that under certain hypothesis they are equal. We are here in the
  hypothesis of his Corollary 1, which states in particular that the Kähler
  different is equal to the different $\FD_{B/A}$.
\end{proof}

Note that we use here a strong result (which applies in much more
generality than what we need). It should be possible to give a simpler
proof. Let us simply add here a more explicit proof of the fact that
the polynomials $J_{B/A}$ and $\theta_{B/A}$ have the same degree
(using a ramification formula of Benson).

\begin{lemma}
With the hypothesis and notations above, we have:
\[\deg (\theta_{B/A}) =\deg (J_{B/A}). \]
\end{lemma}

\begin{proof}
 First we need to recall some notations for graded algebras, see
 \cite[Ch.2.4]{Benson}. If $A$ is a graded algebra, and $M$ a graded
 $A$-module, we denote the usual Hilbert-Poincaré series of $M$ by $\grdim
 M$ (for ``graded dimension''): $\grdim M := \sum_k \dim M_k t^k $. If
 $n$ is the Krull dimension of $A$, we define rational numbers $\deg(M)$
 and $\psi(M)$ by the Laurent expansion about $t=1$:
 \[ \grdim(M)= \frac{\deg(M)}{(1-t)^n} + \frac{\psi(M)}{(1-t)^{n-1}}+
 o\left(\frac{1}{(1-t)^{n-1}}\right).\]

 \medskip 

 Now if we return to our context we can use the following ramification
 formula from \cite[Thm 3.12.1]{Benson}:
 \[ |L:K|\psi(A) - \psi(B) = \frac12 \sum_{\fp \in \Spec_1(B)} v_\fp
 (\FD_{B/A}) \psi (B/\fp) .\] We have $\grdim A= \prod_{i=1}^n
 \frac{1}{1-t^{a_i}}$, so by computing the derivative of $(1-t)^n\grdim A$
 at $t=1$ we get:
 \[ \psi(A)=\frac{1}{\prod_i a_i} \sum_i \frac{a_i-1}{2} \ , \]
 and similarly for $\psi(B)$. As $|L:K|=\prod a_i / \prod b_i$, we obtain:
 \[ |L:K|\psi(A) - \psi(B)= \frac{1}{\prod_i b_i} \sum_i \frac{a_i-b_i}{2}
 = \frac{1}{2\prod_i b_i} \deg J_{B/A}.\]
 On the other hand, for $\fp \in \Spec_1(B)$, if $d$ denotes the degree of a
 homogeneous polynomial~$P$ generating $\fp$, we have
 \[ \grdim B/ \fp= (1-t^d)\prod_i \frac1{1-t^{b_i}} \ ,\]
 so after computation we get
 \[ \psi(B/ \fp)= \frac{d}{\prod_i b_i} \ . \]
 Thus, the r.h.s.\ of the ramification formula becomes 
 \[ \frac1{2\prod_i b_i} \sum_{P \in \Spec_1(B)} v_P
 (\theta_{B/A}) \deg P = \frac1{2\prod_i b_i} \deg \theta_{B/A} \]
 and we can conclude that $\deg J_{B/A}=\deg \theta_{B/A}$.
\end{proof}

\bigskip

As a direct consequence of Propositions \ref{propdiff} and
\ref{propdiffjac}, we obtain the factorisation of $J_{B/A}$.

\begin{theo}
\label{thmjac}
If $A\subseteq B$ is a finite graded
polynomial extension, then we have, with the notations above:
\[ J_{B/A} \doteq \prod_{Q \in \Spram(B)} Q ^{e_Q -1}. \]
\end{theo}

This formula settles the first part of Thm.\ \ref{thmintrojac}. In
section \ref{partwellram} we will define the well-ramified property,
in order to deal with the second part of Thm.\ \ref{thmintrojac}.

\section{Geometric properties}

\label{partgeom}

In this section we recall some definitions and elementary facts about the
ramification locus and branch locus of a branched covering. These will be
used in section \ref{partwellram}, in order to give geometric interpretations of the
well-ramified property (Def.\ \ref{defwellram} and Prop.\ \ref{propwell}).

\subsection{Ramification locus and branch locus}
\label{subpartlocusnew}

Let us define the varieties $U=\Spec A$ and ${V=\Spec B}$, so that to the
extension $A \subseteq B$ (as above), of degree $r$, corresponds an
algebraic quasi-homogoneous finite morphism $f:V \to U$. We
denote by $J$ the Jacobian of $f$.

\begin{propo}
In $V$, we have equalities between:
\begin{enumerate}[(i)]
\item the set of points where $f$ is not étale, \ie zeros of the Jacobian $J$;
\item the union of the sets of zeros of the ramified polynomials of $B$;
\item the set of zeros of the generator $\theta_{B/A}$ of the different
  $\FD_{B/A}$.
\end{enumerate}
\end{propo}

This set is called the \emph{ramification locus} $\Vram$.

\begin{proof}
  (ii)=(iii) comes from Prop.\ \ref{propcaracdiff}, and (iii)=(i) from
  Prop.\ \ref{propdiffjac}.
\end{proof}

\medskip

The following (well-known) proposition gives an upper bound for the
cardinality of the fibers of $f$.

\begin{propo} For all $u\in U$, $|f^{-1}(u)| \leq r$, where $r$ is the
  degree of $f$.
\end{propo}

\begin{proof}
  (After \cite[II.5.Thm.6]{shaf}.) Let $u$ be in $U$, and write
  $f^{-1}(u)=\{v_1,\dots, v_m\}$. One can easily find an element $b$ in
  $B=\CO_V$ such that all the values $b(v_i)$ are distinct.
  
  As $f$ is finite, $B$ is a module of finite type (of rank $r$) over
  $A$. Thus, every element in~$B$ is a root of a unitary polynomial with
  coefficients in $A$, and degree less than or equal to $r$. Let $P\in
  A[T]$ be such a polynomial for $b$, and write $P=\sum_{i=0}^{d}a_iT^i$,
  with $d \leq r$ and $a_d=1$. We have
  \[ \sum_{i=0}^{d}a_ib^i =0 .\] For $j\in \{1,\dots, m\}$, as $f(v_j)=u$,
  specializing the above identity at $v_j$ gives
  \[ \sum_{i=0}^{d}a_i(u) b(v_j)^i =0 .\] So the polynomial $\sum_i a_i(u)
  T^i$ has $m$ distinct roots $b(v_1),\dots , b(v_m)$, and has degree $d$,
  so we obtain $m \leq d \leq r$.
\end{proof}

Points in $U$ whose fiber does not have maximal cardinality are called
\emph{branch points}, and form the \emph{branch locus} of $f$ in $U$:
\[ \Ubr := \left\{ u \in U \ ,\  \left|f^{-1}(u)\right| <r \right\} \ .\]

It is easy to show that $\Ubr$ is closed for the Zariski topology
and is not equal to $U$ (cf. \cite[II.5.Thm.~7]{shaf}), so that
$U-\Ubr$ is dense in $U$.

\begin{propo}
\label{proprambr}
  With the notations above, we have the following equality:
  \[ f(\Vram) = \Ubr \ .\] 

  So $\Vram\subseteq f^{-1}(\Ubr)$. Moreover, the
  restriction of $f$: \[V- f^{-1}(\Ubr)\ \surj \ U-
  \Ubr \] is a topological $r$-fold covering (for the
  transcendental topology).
\end{propo} 

\begin{proof}
  We set $U':= U- \Ubr$ and $V':=V-f^{-1}(\Ubr)$;
  these are Zariski-open.

  \medskip

  First, let $u$ be in $\Ubr$. As $U'$ is dense in $U$, we can
  find a sequence $u^{(k)}$ of elements in $U'$, whose limit is $u$. Let us
  write $f^{-1}(u)=\{v_1,\dots, v_p\}$ (with $p <r$), and
  $f^{-1}(u^{(k)})=\{v_1^{(k)},\dots, v_r^{(k)}\}$. Up to extracting
  subsequences, we can assume that each sequence $(v_i^{(k)})_{k \in \BN}$
  converges towards one of the $v_j$. As $r<n$, we have at least two
  sequences, say $v_1^{(k)}$ and $v_2^{(k)}$, whose limit is the same
  element of $f^{-1}(u)$, say $v_1$. But for all $k$, we have $v_1^{(k)}
  \neq v_2^{(k)}$ and $f(v_1^{(k)})= f(v_2^{(k)})=u^{(k)}$. If $J(v_1)\neq
  0$, this contradicts the inverse function theorem. So $J(v_1)=0$, and $u
  \in f(Z(J))=f(\Vram)$.
  
  \medskip

  By \cite[II.5.3.Cor.~2]{shaf}, if $f$ is unramified at $u$, then for all
  $v\in f^{-1}(u)$, the tangent mapping of~$f$ at~$v$ is an
  isomorphism. That means exactly that $f^{-1}(U- \Ubr) \subseteq
  V-Z(J)$, \ie $f(Z(J)) \subseteq \Ubr$.

  \medskip
  
  Let us fix an element $u$ in $U'$. For all $v$ in $f^{-1}(u)$, we
  know that $J(v)\neq 0$, so, by the inverse function theorem, there
  exists a neighbourhood $E_v$ of $v$ in $V'$ such that the
  restriction $f : E_v \to f(E_v)$ is an isomorphism. Let us write
  $f^{-1}(u)=\{v_1,\dots, v_r \}$. Obviously we can suppose that the
  $E_{v_i}$'s are pairwise disjoint. Then $\Omega_u:=f(A_{v_1})\cap
  \dots \cap f(A_{v_r})$ is a neighbourhood of~$u$ in~$U'$, and for
  $x$ in $f^{-1}(\Omega_u)$ there exists a unique $i$ such that $x$ is
  in~$A_{v_i}$. Thus we get natural map $f^{-1}(\Omega_u) \to \Omega_u
  \times \{1,\dots, r\}$, $x \mapsto (f(x),i)$, which is clearly a
  homeomorphism.
\end{proof}

\subsection{Ramification indices of a branched covering}
\label{subpartbranched}

As explained in Namba's book (\cite[Ex. 1.1.2]{branched}), the map $f$
is more precisely a \emph{finite branched covering} of $U$ (according
to \cite[Def.\ 1.1.1]{branched}), of degree $r$.

\medskip

Because of the inclusion $Z(J)=\Vram \subseteq f^{-1}(\Ubr)$, we know
that the irreducible components of $f^{-1}(\Ubr)$ are the $Z(Q)$, for
$Q \in \Spram(B)$, plus possibly some other components associated to
unramified polynomials. We recall here, for the sake of completeness,
some classical properties, thanks to which the ramification indices
can be interpreted via the cardinality of the fibers.

\begin{propo}[After Namba]
  \label{propfibers}
  Let $u$ be a non-singular point of $\Ubr$, and $v\in f^{-1}(u)$. Then:
  \begin{enumerate}[(a)]
  \item $v$ is non-singular in $f^{-1}(\Ubr)$. In particular,
    there is a unique irreducible component $C_v$ of
    $f^{-1}(\Ubr)$ containing $v$.
  \item There exists a connected open neighbourhood $\Omega_v$ of $v$ such
    that, for the restriction of $f$:
    \[ \widetilde{f}: \Omega_v \to f(\Omega_v) \ , \] for each $u'$ in
    $f(\Omega_v)\cap (U-\Ubr)$, the cardinality of the fiber
    $\widetilde{f}^{-1}(u')$ is the ramification index $e_\fq$ of the ideal
    $\fq \in \Spec_1(B)$ defining $C_v$.
  \end{enumerate}
\end{propo}

\begin{proof}
  We refer to Theorem 1.1.8 and Corollary 1.1.13 in \cite{branched}.
\end{proof}

\section{Well-ramified extensions}
\label{partwellram}
\subsection{The well-ramified property}
\label{subpartwellram}

As above we consider a finite graded polynomial extension $A \subseteq B$,
given by a morphism $f$. We denote by $J$ its Jacobian, and by $e_Q$ the
ramification index of a polynomial $Q\in\Spec_1(B)$.

\begin{defi}\label{defwellram}
We say that the extension  $A \subseteq B$, or the morphism $f$, is
\emph{well-ramified}, if:
\[ \left( J \right) \cap A = \left( \prod_{Q \in \Spram(B)} Q ^{e_Q}
\right) \ . \] In this case we will call the polynomial $\prod
Q^{e_Q}$ the \emph{discriminant} of the extension, and denote it by
$D_{B/A}$.
\end{defi}

Note that if the extension is well-ramified, the quotient
$D_{B/A}/J_{B/A}$ is exactly the product of all ramified polynomials
of the extension.

\medskip

Flagrantly, the definition above makes the second part of Thm.\
\ref{thmintrojac} a tautology. So our terminology might at this point
seem quite mysterious. Actually the remaining of this section will be
concerned with giving several equivalent characterisations of
well-ramified extensions (Prop.\ \ref{propwell}), that ought to make
this terminology (and the usefulness of this notion) much more
transparent.

\subsection{Characterisations of the well-ramified property}

Most of the characterisations given below are very elementary, but are
worth mentioning so as to get a full view of what is a well-ramified
extension. 

\begin{propo}
  \label{propwell}
  Let $A \subseteq B$ a finite graded polynomial extension, and $f : V \to
  U$ its associated morphism. The following properties are equivalent:
  \begin{enumerate}[(i)]
  \item the extension $A\subseteq B$ is well-ramified
    (as defined in \ref{defwellram});
  \item $\displaystyle{\left( \prod_{Q \in \Spram(B)} Q \right) \cap A = \left(
      \prod_{Q \in \Spram(B)} Q ^{e_Q} \right)}$;
  \item the polynomial $\prod_{Q \in \Spram(B)} Q ^{e_Q} $ lies in $A$;
  \item for any $\fp \in \Spec_1(A)$, if there exists $\fq_0 \in
    \Spec_1(B)$ over $\fp$ which is ramified, then any other $\fq \in
    \Spec_1(B)$ over $\fp$ is also ramified;
  \item if $P$ is an irreducible polynomial in $A$, then, as a polynomial
    in $B$, either it is reduced, or it is completely non-reduced, \ie, any of
    its irreducible factors appears at least twice;
  \item $f^{-1}(\Ubr )=\Vram$;
  \item $f(\Vram) \cap f(V- \Vram) = \varnothing$.
  \end{enumerate}
\end{propo}

\begin{proof}
  Let $S:=\prod_{Q \in \Spram(B)} Q $, and $R:=\prod_{Q \in \Spram(B)} Q
  ^{e_Q}$. From Thm.\ \ref{thmjac}, we have also: $J = \prod_{Q \in \Spram(B)} Q
  ^{e_Q-1}$.
  
  We begin with some general elementary facts. We have $(S)\cap
  A=\bigcap_{Q \in \Spram(B)} (Q)\cap A$. For $Q \in \Spram(B)$,
  denote by $\widetilde{Q}$ an irreducible in $A$ such that $Q$ lies
  over $\widetilde{Q}$ (\ie {$(Q)\cap A = \widetilde{Q}$}). As we work
  in UFDs, we get that $(S)\cap A$ is principal, generated by
  \[ \widetilde{S} := \lcm \left(\widetilde{Q} \tq Q\in \Spram(B)\right)\
  .\]

  For $Q \in \Spram(B)$, $Q^{e_Q}$ divides $\widetilde{Q}$, so $R$ divides
  $\widetilde{S}$. Moreover, $J$ divides $R$, so: \\${\widetilde{S} \subseteq
    (R)\cap A \subseteq (J)\cap A}$. Conversely, $S$ divides $J$, so:
  $(J)\cap A \subseteq (S)\cap A= (\widetilde{S})$.
  
  Thus we always have $(J)\cap A = (S)\cap A$, and the statement $(ii)$ is
  just an alternate definition of the well-ramified property:
  $(i)\Leftrightarrow (ii)$. 

  \medskip

  $(iii)\Leftrightarrow (ii)$: we have $(S)\cap A \subseteq (R)$ and $R\in
  (S)$. So $R$ lies in $A$ if and only if $(S)\cap A =(R)$.

  \medskip 

  $(v)\Leftrightarrow (iv)$, since $(v)$ is just a ``polynomial''
  rephrasing of $(iv)$.

  \medskip
  
  $(iv)\Leftrightarrow (iii)$: let us denote by
  $\Spram(A)$ the set of primes $\fp$ in $\Spec_1(A)$ such that there
  exists at least one prime $\fq$ over $\fp$ which is ramified. Then:
  \begin{equation*} R=\prod_{Q \in \Spram(B)} Q ^{e_Q} =\prod_{P \in \Spram(A)}
  \prod_{\substack{Q \in \Spram(B)\\Q \text{ over }P}} Q
  ^{e_Q} \ .\tag{*} \end{equation*} 
  If we suppose $(iv)$, then whenever $P$ is in $\Spram(A)$,
  all the $Q$ over $P$ are ramified. Thus:
  \[R= \prod_{P \in \Spram(A)} \ \prod_{Q \text{\ over\ } P} Q ^{e_Q} =
  \prod_{P \in \Spram(A)} P \] and the polynomial $R$ lies in $A$.

  Conversely, suppose that $R$ lies in $A$. Consider $P$ in $\Spram(A)$,
  and $Q$ in $\Spec_1(B)$ lying over~$P$. Then there exists $Q_0$ in
  $\Spram(B)$ such that $(Q_0)\cap A = (P)=(Q)\cap A$. As $(R)$ is
  contained in $(Q_0)\cap A$, we obtain that $Q$ also divides $R$, so is
  among the factors of the product (*) above. Thus $(Q)$ is ramified, and
  $(iv)$ is verified.

  \medskip
  
  $(vii)\Rightarrow (vi)$: if $v$ lies in $f^{-1}(\Ubr)$, we have
  $f(v)\in \Ubr=f(\Vram)$, so $(vii)$ implies that $v\in
  \Vram$. Thus $f^{-1}(\Ubr) \subseteq
  \Vram$. The other inclusion is from Prop.\ \ref{proprambr}.

  \medskip
  
  $(i) \Rightarrow (vii)$: we have $(\widetilde{S})=(J)\cap A$, so that
  $Z(\widetilde{S}) = f(Z(J))$. If the extension is well-ramified, we
  obtain: $\widetilde{S}\doteq \prod_{Q \in \Spram(B)} Q^{e_Q}$. Thus $J$
  and $\widetilde{S}$ have the same irreducible factors in $B$, which
  implies $Z(J)=f^{-1} (Z(\widetilde{S}))$. So
  $f(\Vram)=Z(\widetilde{S})$, whereas $f(V- \Vram)=U -
  Z(\widetilde{S})$.

  \medskip

  $(vi)\Rightarrow (v)$: suppose that there exists $Q$ in $\Spram(B)$,
  such that $\widetilde{Q}$ has one irreducible factor (in $B$) which is not a
  ramified polynomial, say $M$. Then we can choose $v$ in $Z(M)-Z(J)$. As
  $M(v)=0$, we have $\widetilde{Q}(f(v))=0$ and $\widetilde{S}(f(v))=0$. So
  $f(v)\in Z(\widetilde{S})=f(Z(J))=\Ubr$, which contradicts $(i)$.
\end{proof}

\subsection{Examples and counterexamples}

A fundamental case when the extension is well-ramified is the Galois
case. If $A=B^G$, with $G$ a (reflection) group, all the ramification
indices of the ideals over a prime of $A$ are the same. In our setting, the
well-ramified property is somewhat a weak version of the normality (the
extension is always separable since we work in characteristic zero).

\medskip

Of course the notion is strictly weaker than being a Galois
extension. Take for example the extension $\BC[X^2+Y^3,X^2Y^3]
\subseteq \BC[X,Y]$, with Jacobian $J=6XY^2(X^2-Y^3)$. The ramified
polynomials are $(X)$ (index $2$) and $(Y)$ (index $3$), both above
$(X^2Y^3)$, and $(X^2-Y^3)$ (index $2$), above $((X^2-Y^3)^2)$. So
this extension is well-ramified but not Galois.

\bigskip

A simple example of a not well-ramified extension is:
\[A=\BC[X^2 Y,X^2 +Y] \subseteq \BC[X,Y]=B,\] which is free of rank $4$.
Here the ideal $(X^2 Y)$ in $A$ has two ideals above in $B$: $(X)$ which
is ramified and $(Y)$ which is not, so the extension is not
well-ramified. We compute $\theta_{B/A}=\Jac(f)=X(Y-X^2)=S$ (using the
notations of the proof of Prop.\ \ref{propwell}). So $R=X^2
(Y-X^2)^2$ is not in $A$, actually $(S)\cap A$ is generated by $X^2 Y
(Y-X^2)^2$. 

\subsection{Applications to the Lyashko-Looijenga morphisms}

For any (well-generated) complex reflection group, Bessis introduced
in \cite{BessisKPi1} a morphism called the Lyashko-Looijenga morphism
($\LL$). It gives rise to a finite graded polynomial extensions as
defined in Def.\ \ref{defext}. In \cite{submax}, using the
characterisations of Prop.\ \ref{propwell}, we prove that the 
extensions $\LL$ are always well-ramified. In particular, this implies that
the quotient $D_{\LL}/J_{\LL}$ of the $\LL$-discriminant over the
$\LL$-jacobian is the product of the ramified polynomials of the
extension $\LL$. This is an important structural property, that is
used in \cite{submax} to derive new combinatorial results about
certain factorisations of a Coxeter element in complex reflection
groups.

\subsection*{Acknowledgements.}
This work is part of my PhD thesis \cite{These}. I would like to thank
my advisor David Bessis for his ceaseless support and his help during
this period.

Part of this work has been carried out during a stay at Oxford
Mathematical Institute in February 2009, where I was supported by a
grant of the network ``Representation Theory Across the Channel''. I
thank Bernard Leclerc and Meinolf Geck who administer this grant. In
Oxford I would like to thank all the members of the Algebra Department
for their hospitality, and in particular Raphaël Rouquier for fruitful
discussions, and for suggesting me the proof of Prop. \ref{propdiff}.

I am also grateful to José Cogolludo for his lectures and for
indicating to me Namba's book \cite{branched} about branched
coverings.

\bibliographystyle{alpha}
\bibliography{totalbibli}

\end{document}